%% file: TerminateCartan.tex
\documentclass[10pt]{article}
\usepackage{setspace}
\usepackage{amsmath,amscd}
\usepackage{amsfonts}
\usepackage{verbatim}
\usepackage{amsthm, thmtools}
\usepackage{mathtools}
\usepackage[english]{babel}
\usepackage{amssymb}
\usepackage{blindtext}
\usepackage{leftidx}
\usepackage{mathrsfs}
\usepackage[linesnumbered,ruled]{algorithm2e}
\usepackage{setspace}
\usepackage{tikz}
\usetikzlibrary{positioning}
\usetikzlibrary{matrix}
\usetikzlibrary{arrows}
\usetikzlibrary{calc}
\usepackage{anyfontsize}

\usepackage{enumerate}
\usepackage[shortlabels]{enumitem}
\usepackage{fullpage}
\usepackage{titlesec}
\titleformat*{\section}{\large\bfseries}
\usepackage[autostyle]{csquotes} 
\usepackage{bbm}
\usepackage{bm}

\usepackage{xypic}
\input xy

\numberwithin{equation}{section}

\newtheorem{theorem}{Theorem}[section]

\theoremstyle{definition}
\newtheorem{definition}[theorem]{Definition}
\newtheorem{Ex}[theorem]{Example}
\newtheorem{Rem}[theorem]{Remark}
\newtheorem{algo}[theorem]{Algorithm}

\include{./macrosLegit}
  
\title{\normalsize\bf Termination of Cartan's equivalence method}
\author{\normalsize\OO rn Arnaldsson}
\date{} 

\begin{document}
\maketitle

\input{TC-Introduction}

\input{TC-CartanKuranishi}

\input{TC-symmetry}

\newpage
\bibliographystyle{plain}

\bibliography{./TerminateCartan.bbl}

\end{document}

%% file: macrosLegit.tex
\renewcommand{\o}{\omega}

\newcommand{\Pe}{\Phi_\varepsilon}
\newcommand{\vb}{\mathbf{v}}

\newcommand{\Jinf}{J^{\infty}}

\newcommand{\Gt}{\widetilde{\G}}

\renewcommand{\Xi}{X^i}

\newcommand{\XiK}{X^i_K}
\newcommand{\XiKj}{X^i_{K,j}}

\newcommand{\uaJ}{u^\alpha_J}

\newcommand{\uaJi}{u^\alpha_{J,i}}

\newcommand{\ua}{u^\alpha}

\newcommand{\muiK}{\mu^i_K}

\newcommand{\muiKj}{\mu^i_{K,j}}

\newcommand{\ebo}{\pmb{\eta}}

\newcommand{\upaj}{\Upsilon^\alpha_J}

\newcommand{\upiK}{\Upsilon^i_K}
\newcommand{\upjJ}{\Upsilon^j_J}

\newcommand{\up}{\Upsilon}
\newcommand{\upa}{\Upsilon^\alpha}
\newcommand{\upi}{\Upsilon^i}

\renewcommand\aa{\"a}

\newcommand\OO{\"O}

\newcommand{\jpx}{ j^\infty\varphi\restrict{x}}
\newcommand\jinf{j^\infty}

\newcommand{\intc}{\int\!\!}

\newcommand\ak{\alpha^\kappa}

\newcommand{\gb}{\bar{g}}
\xyoption{all}

\renewcommand{\th}{^{\text{th}}}

\newcommand{\xb}{\bar{x}}

\newcommand{\p}{\varphi}
\newcommand{\ep}{\varepsilon}
\newcommand\Rtwo{R_{j^2\p}}
\newcommand\Rone{R_{j^1\p}}

\newcommand{\Nn}{\mathbb{N}_0^n}
\newcommand{\Nz}{\mathbb{N}_0}

\newcommand{\A}{\mathcal{A}}

\newcommand{\N}{\mathcal{N}}
\newcommand{\E}{\mathcal{E}}
\newcommand{\D}{\mathcal{D}}

\newcommand{\G}{\mathcal{G}}
\renewcommand{\S}{\mathcal{S}}

\newcommand{\T}{\mathcal{T}}

\newcommand{\Rr}{\mathcal{R}}

\newcommand{\R}{\mathbb{R}}

\newcommand{\xbf}{\mathbf{x}}

\newcommand{\dis}{\displaystyle}

\newcommand{\resx}{\restrict{x}}

\newcommand\one{\mathbbm{1}}

\newcommand{\bpm}{\begin{pmatrix}}
\newcommand{\epm}{\end{pmatrix}}
\newcommand{\bbm}{\begin{bmatrix}}
\newcommand{\ebm}{\end{bmatrix}}
\newcommand{\beq}{\begin{equation}}
\newcommand{\eeq}{\end{equation}}
\newcommand{\bex}{\begin{Ex}}
\newcommand{\eex}{\end{Ex}}
\newcommand{\bre}{\begin{Rem}}

\newcommand{\qv}{^{(q)}}

\def\restrict#1{\raise-.3ex\hbox{\scriptsize\ensuremath|}_{#1}}
\def\restrictbig#1{\raise-.3ex\hbox{\large\ensuremath|}_{#1}}
\def\irestrict#1{\raise+.7ex\hbox{\scriptsize\ensuremath|}^{#1}}

\def\comp{\raise 1pt \hbox{$\,\scriptstyle\circ\,$}}

\def\interior{\mathbin{\hbox{\hbox{{\vbox
    {\hrule height.4pt width6pt depth0pt}}}\vrule height6pt width.4pt depth0pt}\,}}

%% file: TC-Introduction.tex
\begin{abstract}
We apply the language of the groupoid approach to Lie pseudo-groups, and the classical Cartan-Kuranishi theorem, to prove that Cartan's equivalence method terminates at involution (or at complete reduction) for constant type problems.  
\end{abstract}

\begin{section}{Introduction}

The equivariant moving frame for pseudo-groups, developed by Olver and Pohjanpelto in \cite{olver05, olver08moving, olver09}, introduced a jet-coordinate based language for these infinite dimensional analogs of Lie groups. The space of jets of pseudo-group elements is a groupoid which Olver and Pohjanpelto used to define Maurer-Cartan forms for Lie pseudo-groups, a key step towards a moving frame for pseudo-groups. In his PhD thesis, \cite{IMF}, the author used the above mentioned works to harmonize Cartan's equivalence method and the equivariant moving frame for pseudo-groups (in cases where both methods are applicable). The result, \emph{involutive moving frames}, enjoys the best each method has to offer; the geometry of Cartan's theory of involution for exterior differential systems and the \emph{recurrence formula} of the moving frame. See \cite{IMF} for many examples of these computational advantages.

Cartan's equivalence method is a powerful tool in differential geometric application for deciding when two geometric structures are equivalent under a change of variables. The geometric structure can be a differential equation (\cite{Bryant95co, Kamran89,
Kamran87}), a variational problem (\cite{Bryant87,  Gardner85, Hsu89,
Kamran91}) a dynamical system (\cite{Gardner83}) or a system of polynomials (\cite{olver90, Olver99}), to name a few examples. The method proceeds by a series of prolongations and projections ended when a system of differential forms is in involution and proving that this is a finite process has been of obvious interest since the inception of the method in \cite{Car53}. However, whether the method terminates at involution (or at complete reduction) can be a confusing matter to gauge from the literature. In \cite[p.2]{bryant95}, written by experts on Cartan's method, the authors state that as far as they know ``... the general result that the construction is a finite process has never really been proven''. In some texts it is claimed that termination of the method should follow by the classical Cartan-Kuranishi completion theorem for sufficiently regular pdes but no references are given for an actual proof (\cite{olver95, sharpe00}). This author has not been successful in tracking down an explicit proof that Cartan's method terminates. 

This paper is an offspring of \cite{IMF}, and in it we present a complete proof of termination of Cartan's method founded upon the groupoid language of \cite{olver05, olver08moving, olver09}. The key difficulty is in connecting the prolonged spaces of Cartan to standard jet spaces where Cartan-Kuranishi holds. The language of groupoids and the theory of Lie pseudo-groups provides a natural bridge between the two worlds. 

In Section \ref{CK} we recall the classic Cartan-Kuranishi theorem on completion of pdes and in Section \ref{CE} we first recall the groupoid approach to pseudo-groups before going carefully through Cartan's equivalence method, and all its twists and turns, to show that it is really just the standard Cartan-Kuranishi completion algorithm written in Cartan's beautiful geometric language.
 
\end{section}

%% file: TC-CartanKuranishi.tex
\begin{section}{Cartan-Kuranishi completion}\label{CK}
In this section we recall the Cartan-Kuranishi completion of (sufficiently regular) pdes. Here, and in the rest of the paper, all diffeomorphisms, differential equations and maps are assumed real-analytic. This is necessary since we need the Cartan-K\aa hler theorem to guarantee local solvability of (well behaved) formally integrable equations, which requires analyticity but all our constructions otherwise work in the smooth category. Let $\E$ be the trivial bundle $\R^n\times\R^m\to\R^n$ in coordinates $x\in \R^n$, $u\in \R^m$, and let $J^q(\E)$ be the space of $q$-jets of sections of $\E$ for $0\leq q\leq\infty$. We denote the elements of $J^q(\E)$ by $j^qu\resx$ or $(x,u\qv)$. 

Consider a $q\th$ order differential equation on $\E$, \beq\label{F}F(x,u\qv)=0.\eeq We denote the set of points in the $q\th$ order jet space that satisfy the equation (\ref{F}) by $\Rr_q\subset J^q(\E)$. 

\bre
We shall refer both to the equations (\ref{F}) and the subset $\Rr_q\subset J^p(\E)$ that they determine as a differential equation. 
\end{Rem}

We can prolong the set (\ref{F}) of equations to order $q+1$ by adjoining to (\ref{F}) all the equations 
\beq\label{F1}
D_{1}F=0,\ldots, D_{n}F=0,
\eeq
and obtain the set $\Rr_{q,1}\subset J^{q+1}(\E)$, where $D_i$ is the total derivative operator on $\Jinf(\E)$,
\[
D_i=\frac{\partial}{\partial x^i}+\sum_{|J|\geq0}\uaJi\frac{\partial}{\partial \uaJ}.
\] 
Note that every local solution to (\ref{F}) must also satisfy the prolonged equation. An integrability condition appearing when going from $\Rr_q$ to $\Rr_{q,1}$ is an equation of order at most $q$ that is an algebraic consequence of the equations (\ref{F1}) but not an algebraic consequence of the equations (\ref{F}). That is, it is a new equation of order (at most) $q$ that solutions to (\ref{F}) must satisfy and should be added to (\ref{F}). The set of points in $\Rr_q$ that also satisfy these integrability conditions is denoted $\Rr^{(1)}_{q,1}$. We can describe the set $\Rr^{(1)}_{q,1}$ using the canonical projections $\pi^{p}_t:J^{p}(\E)\to J^{t}(\E)$, $0\leq t\leq p\leq \infty$, between the jet spaces:
\[
\Rr^{(1)}_{q,1}=\pi^{q+1}_q(\Rr_{q,1}).
\]
Note that the presence of integrability conditions is equivalent to the condition that $\Rr^{(1)}_{q,1}\subsetneq \Rr_q$.

More generally, adjoining all prolongations of (\ref{F}) of order $t$, $D_JF=0$, $J\in\Nn$ with $|J|=t$, we arrive at the set $\Rr_{q,t}\subset J^{q+t}(\E)$. We denote the projection $\pi^{q+t}_{q+t-s}(\Rr_{q,t})$ by $\Rr^{(s)}_{q,t}\subset\Rr_{q,t-s}$. The differential equation $\Rr_q$ is \emph{formally integrable} if $\Rr^{(s)}_{q,t}=\Rr_{q,t-s}$ for all $t\geq s$.

For a $q\th$ order equation $\Rr_q\subset J^q(\E)$, given by a system (\ref{F}), we shall write $\Rr_\infty$ for the set of points in $\Jinf(\E)$ that satisfy (\ref{F}) and \emph{all} its prolongations (and hence all integrability conditions of all orders). 

\begin{Rem}\label{rem:degenerate}
To prevent too much degeneracy in our differential equations, we assume that for all differential equations we encounter that the full system $\Rr_\infty\subset\Jinf(\E)$ is fibered over all of $\R^n$, i.e. the system does not impose any restrictions on the independent variables alone. We shall refer to such differential equations as \emph{genuine differential equations}.
\end{Rem}

\begin{Rem}\label{rem:qcomplete}
If $\Rr_q$ is given by equations (\ref{F}) and some equations are of order strictly less that $q$ we ``complete'' (\ref{F}) to an equivalent $q\th$ order system in the following way. Let $\S_{q-1}\subset J^{q-1}(\E)$ be the system determined by the equations in (\ref{F}) of order strictly less that $q$. Prolong all these equations to order $q$ to obtain $\S_q$ and replace $\Rr_q$ by $\S_q\cap\Rr_q$. Now repeat this process, prolonging each equation of order $<q$ to order $q$, and so forth, until we no longer obtain new equations. This new system has an important property; if $F_j(x,u^{(q-s)})=0$ is an equation in $\S_q\cap\Rr_q$, then any prolongation of it, $D_JF_j=0$, $|J|\leq s$, appears as an equation in $\S_q\cap\Rr_q$. Assume $\Rr_q$ is complete in the above way and consider a point $(x_0,u\qv_0)\in\Rr_q$ and let $\p$ be a local solution to the equations of order exactly $q$ only, but such that $j^q\p\restrict{x_0}=(x_0,u\qv_0)$. Then we have, for any equation $F_j(x,u^{(q-1)})=0$ in $\Rr_q$ of order $q-1$ that, for all $i$,
\[
D_i\left(F_j(j^{q-1}\p\resx)\right)=0\quad\text{and}\quad F_j(j^{q-1}\p\restrict{x_0})=0.
\]
But this means that $\p$ is a local solution to all equations in $\Rr_q$ of order $q-1$, $F_j(j^{(q-1)}\p\resx)=0$. Similarly, $\p$ is a local solution to all the lower order equations in $\Rr_q$. This means that for questions of local solvability it is sufficient to consider systems $\Rr_q$ that are determined by equations of order exactly $q$ only. We assume all differential equations $\Rr_q$ have been completed in this way. 
\end{Rem}


\begin{Rem}\label{rem:hilbert}
Let $\Rr_q$ be a differential equation and let $\jinf u\resx\in\Rr_\infty$. An integrability condition appearing in $\Rr_{q,t}^{(s)}\subsetneq \Rr_{q,t-s}$ enlarges the symbol module of the original $\Rr_q$ at $\pi^\infty_q(\jinf u\resx)$ and by the Hilbert basis theorem we will eventually stop finding such integrability conditions and the symbol module stabilizes at each point. To make sure that this happens at the same time for all points in $\Rr_\infty$ we make the following regularity assumption on all differential equations in this paper: For all $s\leq q+t$, $\Rr_{q,t}^{(s)}$ is a \emph{submanifold} in $J^{q+t-s}(\E)$ and there is a $p^*$ such that the fibers of
\[
\Rr_{q,t}^{(s)}\to\pi^{q+t-s}_{p^*}\left(\Rr_{q,t}^{(s)}\right)
\]
have constant dimension for all $s,t$ such that $q+t-s\geq p^*$. This will prevent the degeneracy mentioned above since, above order $p^*$, the symbol modules have the same homogeneous dimensions at different points in the differential equation. 

This will also guarantee that the completion process of a differential equation to some order $q$ from Remark \ref{rem:qcomplete} will terminate with a system with the desired properties as long as $q$ is at least as large as $p^*$.  
\end{Rem}

\begin{Rem}\label{rem:solved}
We shall, without loss of generality, that all of our equations $\Rr_q$ have been written in \emph{solved form}. This means that each defining equation has the form $\ua_K=F^\alpha_K(x,u^{(q)})$ and no $\ua_K$ appearing in a left hand side of such an equation appears in a right hand side. The $\ua_K$ are the \emph{principal derivatives}, while their complementary jet coordinates are \emph{parametric derivatives}. For regular systems (see Definition \ref{regular} below) this can always be achieved locally, and since our interest is only in local solvability, this is no real restriction. 
\end{Rem}

As it stands, the regularity hypothesis described in the above remark are difficult to check, but we can give a necessary condition based on \emph{reduced Cartan characters} which we now recall. Let $\Rr_q\overset{\iota}{\hookrightarrow} J^q(\E)$ be a differential equation and denote the standard contact forms on $\Jinf(\E)$ by $\upaj:=d\uaJ-u^\alpha_{J,i}dx^i$, $J\in\Nn$. The contact codistribution on $\Rr_q$ is generated by the restriction (or pull-back) to $\Rr_q$ of
\[
\{\upaj~|~|J|<q\}\quad\overset{\iota^*}{\to}\quad\{\upaj\restrict{\Rr_q}~|~|J|<q\}
\]
and we define the projection $\gamma_q:T^*\Rr_q\to T^*\Rr_q$ onto to subspace generated by the $d\ua_K$ for $|K|=q$. The reduced Cartan characters of $\Rr_q$ are computed as follows. First maximize (over all $(a_1^1,\ldots, a_1^n)\in\R^n$) the rank of the set 
\[
\gamma_q\left(\{(\sum_ia_1^iD_i\interior d\upaj\restrict{\Rr_q}~|~|J|=q-1\}\right)
\]
of one-forms, where the $\upaj$ are all restricted to $\Rr_q$. This gives the first reduced Cartan character $s^q_1$. Having computed the first $k$ reduced Cartan characters $s^q_1,\ldots, s^q_k$, we maximize the rank of the one-forms
\[
\gamma_q\left(\{(\sum_ia_1^iD_i\interior d\upaj\restrict{\Rr_q}~|~|J|=q-1\}\cup\dots\cup\{(\sum_ia_{k+1}^iD_i\interior d\upaj\restrict{\Rr_q}~|~|J|=q-1\}\right)
\]
over all collections of  $k+1$ vectors $(a_1^1,\ldots, a_1^n),\ldots, (a_{k+1}^1,\ldots, a_{k+1}^n)$ in $\R^n$ to obtain  the number $r_{k+1}$. The $k+1$ reduced Cartan character of $\Rr_q$ is 
\beq\label{eq:sqk}
s^q_{k+1}=r_{k+1}-s^q_k-\dots-s^q_1.
\eeq

The \emph{symbol} of $\Rr_q$ is involutive if the reduced Cartan characters satisfy Cartan's test, i.e. if the fiber dimension of the projection $\pi^{q+1}_q:\Rr_{q,1}\to\Rr_q$ is equal to the sum
\[
\sum is^q_i.
\]
We denote this fiber dimension by $r^{q+1}$, and note that it is equal to the number of \emph{parametric derivatives} of order $q+1$ in the system $\Rr_{q,1}$ (for any choice of principal/parametric derivatives). Since, in the process of completing a differential equation to involution one must compute reduced Cartan characters we further restrict to differential equations having \emph{constant} reduced Cartan characters above a certain order.

\begin{definition}\label{regular}
We say that $\Rr_q$ is \emph{regular} if for all $s\leq q+t$, $\Rr_{q,t}^{(s)}$ is a submanifold in $J^{q+t-s}(\E)$ and there is a $p^*$, called the \emph{regularity order} of $\Rr_q$, such that the systems $\Rr_{q,t}^{(s)}$, for all $s,t$ such that $q+t-s\geq p^*$, have constant reduced Cartan characters $s^{q+t-s}_i$.
\end{definition}

It is easy to see that this notion of regularity guarantees the desired properties of Remark \ref{rem:hilbert}. 

\bex\label{ex:abel1}
Consider the second order system, $\Rr_2$, of differential equations for maps 
\[
(x,y,u)\mapsto(X(x,y,u), Y(x,y,u), U(x,y,u))
\]
given by the first order equations
\beq\label{eq:abel1}
\begin{gathered}
X=x,\quad Y=y,\quad U=u+xU_x+yU_y,\\
X_x=Y_y=U_u=1,\quad X_y=X_u=Y_x=Y_u=0
\end{gathered}
\eeq
and the trivial second order equations
\[
X_{ij}=Y_{ij}=U_{ij}=0,\quad\text{for all}~ i,j\in\{x,y,u\}.
\]
Obviously $\Rr^{(1)}_2$, given by the equations (\ref{eq:abel1}), is a manifold but the symbol of $\Rr^{(1)}_2$ is non regular since for $x=y=0$ the equation 
\[
U=u+xU_x+yU_y
\]
drops in degree. However, this system is regular with regularity order $p^*=2$, above which these problems obviously do not arise. 
\eex

\begin{definition}
A formally integrable differential equation $\Rr_q\subset\Jinf(\E)$ is \emph{involutive} if it has constant $q\th$ order reduced Cartan characters and its symbol is involutive at all points in $\Rr_q$.  
\end{definition}
By the Cartan-K\aa hler theorem involutive equations are locally solvable and a solution depends on $s^q_i$ free functions of $i$ variables.

At first sight, formal integrability (and hence involutivity) looks like a condition that can only be affirmed by checking \emph{all} prolongations of $\Rr_q$, but fortunately there exists a \emph{finite} process for obtaining a formally integrable equation from the initial set $\Rr_q$ if $\Rr_q$ is regular. The following theorem is the key (see \cite{s10}).

\begin{theorem}\label{forminv}
Let $\Rr_q$ be a $q\th$ order, regular, differential equation whose symbol is involutive and assume that $\Rr^{(1)}_{q,1}=\Rr_q$, then $\Rr_q$ is formally integrable.
\end{theorem}

Graphs of solutions to $\Rr_q$ annihilate the contact forms $\upaj\restrict{\Rr_q}$, $|J|<q$, and so must also annihilate their exterior derivative. When looking for integrability conditions in $\Rr^{(1)}_{q,1}$, the ``intrinsic'' alternative to taking total derivatives $D_i$ of the defining equations of $\Rr_q$ is to compute the exterior derivative of all $\upaj\restrict{\Rr_q}$, $|J|<q$. Since the systems arising in Cartan's equivalence method, to be treated in the next section, are all order \emph{one} we describe this intrinsic method in this simpler case. Assume, then, that we have a first order system $\Rr_1$. The contact codistribution on $\Rr_1$ is generated by the zero order contact forms $\upa$.  We have  
\beq\label{eq:dup}
d\upa\restrict{\Rr_1}=\sum_id\ua_{i}\restrict{\Rr_1}\wedge dx^i,
\eeq
where, to compute $d\ua_{i}\restrict{\Rr_1}$ we replace each principal derivative $\ua_k$ by the corresponding right hand side $F^\alpha_k$ (see Remark \ref{rem:solved}). In $dF^\alpha_k(x,u^{(1)})$ we obtain a linear combination of $du^\beta_l$, where $u^\beta_l$ are parametric, $dx$'s and $d\ua$. The last of these we write
\[
d\ua=\upa+\sum\ua_idx^i
\]
as we are only interested in $d\upa\restrict{\Rr_1}$ modulo contact forms on $\Rr_1$. We introduce second order jet coordiantes $z^\beta_{l,i}$ (where we do no longer identify $z^\beta_{l,i}$ and $z^\beta_{i,l}$) and set 
\beq\label{eq:uz}
du^\beta_l=\left(du^\beta_l-\sum_iz^\beta_{l,i}dx^i\right) + \sum_iz^\beta_{l,i}dx^i,
\eeq
where the part in parenthesis completes the $\upa$ to a basis of contact forms on $\Rr_{1,1}$. Plugging this into (\ref{eq:dup}) and setting to zero we find a whole host of \emph{linear} (non-homogeneous) equations for the new jet coordinates $z^\beta_{l,i}$, we shall call \emph{the absorbtion equations}. Next we solve for as many $z$'s as we can (we call those we solve for principal) and the $z$'s we do \emph{not} manage to solve for (which we call parametric) will serve to parametrize, along with $(x,u^{(1)})$, $\Rr_{1,1}$. If, after solving some equations for $z$'s, we find an equation depending only on the 1-jets $(x,u^{(1)})$ we have found an integrability condition which must be adjoined to $\Rr_1$. If this does not happen, we compute reduced Cartan characters and perform Cartan's test for involutivity where the fiber dimension $r^2$ is equal to the number of free $z$'s in the absorbtion equations. If Cartan's test is not satisfied, we repeat this process starting with $\Rr_{1,1}$ instead of $\Rr_1$. The contact codistribution on $\Rr_{1,1}$ in the coordinates $(x,u^{(1)}, z)$ is given by the $\up$'s along with
\[
du^\beta_l- z^\beta_{l,i}dx^i
\]
where we replace principal $z$'s with parametric $z$'s. 

The above procedure is of course well known and it is easy to see it is equivalent to the ``extrinsic'' approach of prolonging equations by using the total derivatives $D_i$. Also, as a consequence of $d(dx^i)=0$ we find that indeed $z^\beta_{i,l}=z^\beta_{i,l}$. But the coordinate one-forms $dx^i$ and the standard contact forms $\upa=d\ua-\ua_i dx^i$ are just one choice of forms with which to work in $\Rr_1$. The above process could just as well be performed for some invertible linear combination of the $\upa$. Similarly, instead of the coordinate forms $dx^1,\ldots, dx^n$, we could choose one-forms $\o^1, \ldots, \o^n$ as long as at each point $p\in\Rr_1$ they generate the same subspace of $T^*_p\Rr_1$ as the horizontal forms $dx^1,\ldots, dx^n$. The expansion (\ref{eq:uz}) would then be
\[
du^\beta_l-\sum_iz^\beta_{l,i}\o^i,
\]
but everything else goes through as before. For a thourough investigation of these matters, we refer to \cite{s10}.

Here then is the classical algorithm for completing sufficiently regular differential equations to involution. 


\begin{algo}\label{ckalgo}
~
\begin{enumerate}[(a)]
\item First check if $\Rr^{(1)}_{1,1}=\Rr_1$ by computing all $d\upa$ on $\Rr_1$. If not, then replace $\Rr_1$ by $\Rr_{1,1}^{(1)}$ and start over.
\item When we stop getting integrability conditions in (a), set up the relevant matrices, and compute the reduced Cartan characters, $s^1_i$. If they are non constant we repeat step (a), this time with $\Rr_{1+1}:=\Rr_{1,1}$ instead of $\Rr_1$. 
\item  Once we have locally constant reduced Cartan characters in step (b), for some system $\Rr_p$, we count the number of parametric derivatives of order $p+1$ in $\Rr_{p,1}$, $r^{p+1}$, and check whether $r^{p+1}=\sum is^p_i$. If successful, we know that there exist coordinates in which $\Rr_{p}$ is involutive in the open neighborhood from step (b). If this fails, we go back to (a) with $\Rr_p$ replaced by $\Rr_{p,1}$. 
\end{enumerate}
\end{algo}

Every integrability condition found during Algorithm \ref{ckalgo} enlarges the symbol module of the equation and so, by the Hilbert basis theorem and regularity, we will eventually stop finding integrability conditions. By regularity we will eventually move past step (b) and Cartan's test in step (c) will succeed eventually by, for example, the existense theorem for $\delta$-regular coordinates, \cite[Theorem 4.3.15]{s10}. The result is, after changing variables into the $\delta$-regular ones, an involutive differential equation and the reduced Cartan characters contain information on the order of freeness of a general solution to $\Rr$, \cite{s10}. For example, if $\Rr_q$ is involutive and contains only equations of order exactly $q$, then $s^q_i$ is the number of free functions of $i$ variables that the general solution of $\Rr_q$ will depend on.

\end{section}

%% file: TC-symmetry.tex
\begin{section}{Termination of Cartan's equivalence method}\label{CE}

Cartan's equivalence method is a powerful tool in differential geometry for finding local invariants of geometric objects. The method proceeds in a very similar fashion as our completion algorithm \ref{ckalgo} with some twists along the way (see \cite{Gardner89, olver95} for nice expositions of the method), and it has been of long-standing interest to rigorously prove that it terminates at involution although it has been suspected/assumed to follow, in one way or another, from the Cartan-Kuranishi theorem, i.e. termination of Algorithm \ref{ckalgo}. The method has an interesting history and it has puzzled many mathematicians, especially the original exposition of Cartan, \cite{Car53}. In the introduction to his book on the method, Robert Gardner \cite{Gardner89}, offers his explanation for these difficulties. 

\begin{displayquote}
\emph{The reason, I believe, was that the method left too much apparent freedom in the way part of the constructions were done. In particular, the process of Lie algebra compatible absorption of torsion and the process of reduction of structure group were not laid out in any systematic way... After thinking about this method for another twenty years... I realized that mixing Cartan's original method with the concept of principal components in Cartan's theory of R\'ep\`ere Mobile led to an algorithmic way to execute Cartan's method.}
\end{displayquote}

Indeed, in the author's thesis, \cite{IMF}, it was shown that Cartan's method and the modern formulation of Cartan's R\'ep\`ere Mobile, \cite{olver08moving}, are really different sides of the same coin. There turned out to be a tremendous \emph{computational} advantage to incorporating the moving frame into Cartan's method as it invites the application of the \emph{recurrence formula} to drastically cut down the necessary calculations of Cartan. 

In this section, though, by carefully comparing the different steps in Cartan's method to the steps of Algorithm \ref{ckalgo}, we shall show that both routines find the same integrability conditions and compute the same reduced Cartan characters and hence terminate at the same time. We shall first introduce the basic language of Lie pseudo-groups as developed in \cite{olver05} before introducing Cartan's method and proving its termination as a consequence of the termination of Algorithm \ref{ckalgo}.

\begin{subsection}{Lie pseudo-groups}
So far we have studied differential equations in an independent variable $x\in \R^n$ and a dependent variable $u\in \R^m$. In the following we shall consider local diffeomorphisms of $\R^n$ and so $x,u\in \R^n$. At the risk of some early confusion, it will be convenient to denote the dependent variables with the capitalized $X$, instead of $u$. 

Consider the jet bundle $\Jinf(\R^n\times\R^n)$ for sections of the trivial bundle $\R^n\times\R^n\overset{\sigma}{\to}\R^n$. Let $\D(\R^n)$ denote the collection of all local diffeomorphisms of $\R^n$ and let $\D_\infty(\R^n)\subset\Jinf(\R^n\times\R^n\to\R^n)$ be the subbundle of all infinite jets of these. We shall sometimes (when it is clear what the $\R^n$ is) drop the mention of $\R^n$ and simply write $\D$ and $\D_\infty$ instead of $\D(\R^n)$ and $\D_\infty(\R^n)$. Similarly, we denote by $\D_p(\R^n)$ the set of $p$-jets of transformations from $\D(\R^n)$. For local coordinates $x$ on $\R^n$ we have the induced jet coordinates $(x,X,\ldots, \XiK,\ldots)$ on $\D_\infty$ (and by truncation on $\D_p$). That is, for a local diffeomorphism $\p$, we have $\jpx=(x,X,\ldots, \XiK,\ldots)$, where $$\XiK=\frac{\partial^{|K|} \p}{\partial x^{K}}(x),\quad K\in\Nn.$$ 
 
The collection $\D$ forms a \emph{pseudo-group}, since if $\p\in\D$ then $\p^{-1}\in\D$ and the composition of two local diffeomorphisms is again a diffeomorphism \emph{whenever the composition can be defined}. As emphasized by Ehresmann, \cite{Ehresmann53}, each set $\D_p\subset J^p(\R^n\times\R^n)$ carries a \emph{groupoid structure}; we define the source and target of a $p$-jet $j^p\p\resx=(x,X,\ldots,X^i_K,\ldots)\in\D_p$, $|K|\leq p$, as
\[
\sigma(j^p\p\resx)=x\quad\text{and}\quad \tau(j^p\p\resx)=X,
\]
respectively. The groupoid multiplication of $j^p\p\resx$ and $j^p\psi\restrict{X}$, where $\tau(j^p\p\resx)=\sigma(j^p\psi\restrict{X})$, is defined as
\[
j^p(g\comp f)\resx,
\]
where $f$ and $g$ are functions in $\D$ having the $p$-jets $j^p f\resx=j^p\p\resx$ and $j^pg\restrict{X}=j^p\psi\restrict{X}$. This definition does not depend on the choice of $f$ and $g$ as can be seen from the chain rule and the resulting combinations of the jets of $\p$ and $\psi$ may be given in terms of Bell polynomials via the general F\`aa-di-Bruno formula, \cite{olver00can}. We write the groupoid operation as $j^p\psi\restrict{X}\cdot j^p\p\resx$. The source and target maps provide each $\D_p$ with a double fibration,

\[
\begin{tikzpicture}
  \node (A0) at (0,0) {$\D_p$};
  \node (A1) at (240:2cm) {$\R^n$};
  \node (A2) at (300:2cm) {$\R^n$.};
  \draw[->,font=\scriptsize]
  (A0) edge node[left] {$\sigma$} (A1)
  (A0) edge node[right] {$\tau$} (A2);

\end{tikzpicture}
\]
 
\begin{definition}\label{def:pseudo}
A \emph{Lie pseudo-group}, $\G$, of local transformations of $\R^n$ is a sub-pseudo-group of $\D$ that is determined by a set of formally integrable differential equations called the \emph{defining equations} that are regular in the sense of Definition \ref{regular} (and can hence be completed to their involutive form). 
\end{definition}

\bex
The pseudo-group of local transformations of $\R^2$ of the form
\[
(x,u)\mapsto (X,U)=\left(f(x), \frac{u}{f'(x)}\right),
\]
where $f:\R\to\R$ is invertible, are a Lie pseudo-group with determining equations
\[
X_u=0,\quad X_xU_u=1.
\]
These are easily found to be regular.
\eex

We denote the collection of transformations making up the pseudo-group by $\G$ while subscripts will indicate the set of groupoid elements, e.g. $\G_\infty$ is the set of infinite jets of transformations from $\G$. Each $\G_p$, $0\leq p \leq \infty$, is a sub-groupoid of $\D_p$. Let $\G$ be a Lie pseudo-group determined by the formally integrable equations 
\begin{equation}\label{def}
F(x,X^{(q)})=0,
\eeq
where $(x,X^{(q)})=(x,X,\ldots, \XiK,\ldots)$, $|K|\leq q$, denotes all jets up to order $q$. Let $\Pe$ be a one parameter family of diffeomorphisms from $\G$. The flow $\ep\mapsto \Pe(x)$ through points $x\in\R^n$ generates a vector field
\beq\label{vid}
\vb(x)=\zeta^i(x)\frac{\partial}{\partial x^i}
\eeq
in the \emph{Lie algebroid} $\A$ of $\G$ of local vector fields on $\R^n$. The components of $\vb$ satisfy the linearization of (\ref{def}) at the identity section $\mathbbm{1}$:
\beq\label{lin}
L(x,\zeta^{(q)})=\frac{\partial F(x,X^{(q)})}{\partial \XiK}\restrictbig{\mathbbm{1}}\zeta^i_K=0.
\eeq
It is a simple matter, under the present regularity hypothesis, to show that for transitive Lie pseudo-groups, (\ref{lin}) are locally solvable whenever (\ref{def}) are. 

A local transformation $\psi\in\G$ acts on the set of jets $\jpx\in\G_\infty$ with $x\in\text{dom}~\psi$ via the right groupoid product by
\beq\label{Rcdot}
R_{\psi}\cdot \jpx=j^\infty\p\resx\cdot j^\infty\psi^{-1}\restrict{\psi(x)}.
\eeq
Notice that this action leaves the target coordinate $X$ invariant, and that it projects to the canonical action 
\[
x\mapsto \psi(x),
\]
on the source coordinates since $\sigma(R_{\psi}\cdot \jpx)=\psi(x)$.

In \cite{olver05} a basis, $\{\muiK\}_{K\in\Nn, 1\leq i\leq n}$, for the contact co-distribution on $\D_\infty$, that is \emph{invariant} (under the action (\ref{Rcdot}) of $\D$) was constructed. Naturally, these contact forms are called the Maurer-Cartan forms of the pseudo-group $\D$. The form $\muiK$ agrees with the standard contact form $\upiK=dX^i_K-X^i_{K,j}dx^j$ on the identity section $\mathbbm{1}$ of $\D_\infty\to\R^n$, and each $\muiK$ is a linear combination of $\upjJ$ for $|J|\leq |K|$ (and conversely). When we restrict the Maurer-Cartan forms to a sub-groupoid $\G_\infty\subset\D_\infty$, we obtain certain linear dependencies among the $\muiK$. The important discovery, made in \cite{olver05}, is that these are given by
\beq\label{linmc}
\frac{\partial F(x,X^{(q)})}{\partial \XiK}\restrictbig{\widetilde{\mathbbm{1}}}\muiK=0,
\eeq
where restriction to $\widetilde{\mathbbm{1}}$ means first restricting to the identity section and then replacing all source coordinates $x$ by target coordinates $X$.

\bre\label{strrem}
The structure equations, or the formulas for $d\muiK$ on $\D_\infty$, were worked out in \cite{olver05}. They are 
\beq\label{structureeq}
d\muiK=\sum_{1\leq j\leq n}\o^j\wedge\muiKj+\sum_{\substack{L+M=K \\ |M|\geq1}}\binom{K}{L}\sum_{1\leq j\leq n}\mu^i_{L,j}\wedge\mu^j_M.
\eeq
Where $L+M$ is the componentwise addition of multi-indices in $\Nz^n$ and
\[
\binom{K}{L}=\frac{K!}{L!M!}.
\] 
These can then be pulled back to $\G_\infty\subset\D_\infty$ and, using (\ref{linmc}), gives the structure equations of the Maurer-Cartan forms on $\G_\infty$.
\end{Rem}

\end{subsection}

\begin{subsection}{Cartan's equivalence method}

Cartan's equivalence method concerns Lie pseudo-groups given by systems of PDE, for local diffeomorphisms $\p$ of $\R^n$, having the special form

\beq\label{eq:cartanpde}
\p^*\ebo=g\ebo.
\eeq

In the above equation, $\ebo$ is a fixed column vector of $n$ \emph{coframe elements} $\eta^1, \ldots, \eta^n$, $g$ is an element of some subgroup $G$ of the general linear group GL$(n,\R)$, and $g\ebo$ is the obvious ``matrix product''. In all known (at least to me) applications, $G$ is a closed algebraic subgroup. In any case, we require $G$ to be a sub Lie group of GL$(n,\R)$. Notice that (\ref{eq:cartanpde}) indeed defines a pseudo-group since if $\p^*\ebo=g\ebo$ and $\psi^*\ebo=h\ebo$, we have
\[
(\p^{-1})^*\ebo=g^{-1}\ebo,\quad\text{and}\quad (\p\comp\psi)^*\ebo=gh\ebo,
\]
whenever $\p$ and $\psi$ can be composed.

When (\ref{eq:cartanpde}) is written out in coordinates it determines a first order PDE for $\p$ but in applications the presentation (\ref{eq:cartanpde}) arises as that capturing the invariance of some geometric structure and is most natural in dealing with such problems. Indeed, the geometry of these differential equations, or rather, of the structure groups $G$, can allow for their completion to involution in problems that are hopeless to attack with direct methods (such as equivalence problems in high dimensions). A coframe $\ebo$ along with a choice of subgroup $G$ is called a $G$-structure on $\R^n$. Denote by $\G$ the pseudo-group of local diffeomorphisms satisfying (\ref{eq:cartanpde}).

\bex\label{ex:inv}
As an example of the reduction in complexity equations of the form (\ref{eq:cartanpde}) provide, compared to the standard jet-coordinate formulation, consider the \emph{divergence equivalence problem} for first order Lagrangians on the line (cf. \cite{olver95} for more on this problem). Let $\intc L(x,u,p)dx$ be a first order Lagrangian in one independent and one dependent variable, where we denote $u_x=p$. A \emph{contact transformation}, $\p$, on the space, $\R^3$, of $(x,u,p)$, preserves this Lagrangian \emph{up to total divergence} if and only if $\p^*\ebo=g\ebo$, where $\ebo$ is the coframe
\[
\eta^1=dx,\quad \eta^2=du-pdx,\quad \eta^3=-\tilde{E}dx+L_{pp}dp,
\]
where $\tilde{E}=L_u-L_{px}-pL_{pu}$ are the \emph{truncated} Euler-Lagrange equations, and $g\in G$, where 
\[
G=\left\{\bbm a_1 & a_2 & a_3 \\ 0 & a_4 & 0 \\ 0 & a_5 & \dfrac{1}{a_4}\ebm~ \middle|~a_1a_4\neq 0 \right\}.
\]

Note that $\ebo$ is a coframe if $L_{pp}\neq 0$, which we must assume. The succinct condition $\p^*\ebo=g\ebo$, in jet coordinates $j^1\p\resx=(x,u,p, X,U,P,X_x,\ldots,P_p)$, is quite a bit more involved and harder to work with. Writing $z=(x,u,p)$ and $Z=(X,U,P)$, the condition $\p^*\ebo=g\ebo$ is equivalent to the system
\begin{align*}
0&=-p \left(P X_u+U_u\right)+P X_x+U_x\\
0&=P X_{p}+U_{p}\\
L_{pp}(Z)&=L_{pp}(z) \left(P_{p}+\tilde{E}(Z)X_{p}\right)(P X_u+U_u)\\
0&=-\tilde{E}(z) P_{p}-p \left(P_u+\tilde{E}(Z) X_u\right)+P_x-\tilde{E}(z)\tilde{E}(Z)
   X_{p}+\tilde{E}(Z) X_x.
\end{align*}
In general, the group parameters of $G$ will be complicated expressions of the 1-jets $j^1\p\resx$, but do serve to parametrize, along with the source and target coordinates $(x,X)$, the first order jet space, $\G_1$, of the pseudo-group of all equivalence maps $\p$.
\eex

Consider a $G$ structure with base coframe
\beq\label{Adx}
\ebo=A(x) d\xbf,
\eeq
where $d\xbf$ is the column vector of the coordinate coframe on $\R^n$, $dx^1, \ldots, dx^n$, and $A:\R^n\to G$ is a smooth map. Assume $G$ is $r$ dimensional and parametrized by the real numbers $a_1,\ldots, a_r$. Cartan's equivalence method completes the system $\p^*\ebo=g\ebo$ to involution in a similar fashion as Algorithm \ref{ckalgo}, but has some (rather brilliant) twists along the way. We shall now closely compare these two routines and demonstrate that they are equivalent in the cases where Cartan's method works. Consider Algorithm \ref{ckalgo} and notice that instead of using the standard basis of contact forms $\Upsilon^i_K=d\XiK-\XiKj dx^j$ on $\Jinf(\R^n\times\R^n)$ we could have chosen any basis $\{\Theta_t\}_{t\in \T}$, for an index set $\T$, with filtration 
\[
\T_0\subset\T_1\subset\T_2\subset\dots\subset\T,
\] 
such that, for each $q\in \Nz$, $\{\Theta_t\}_{t\in\T_q}$ is a basis for the contact codistribution on $J^{q+1}(\E)$. This is simply because, for each $q\geq 0$, the set $\{\upiK\}_{|K|\leq q}$ is an invertible linear combination of contact forms in $\{\Theta\}_{t\in\T_{|K|}}$, and therefore, we shall find the same numbers $r^{q+1}$ of parametric derivatives and reduced Cartan characters $s^q_i$ as if we had used the standard basis. The Maurer-Cartan forms on $\D$ is one such collection, and Cartan's equivalence method explicitly constructs, order by order, another such invariant collection.

Cartan's equivalence method starts out, not by computing the exterior derivatives of the standard zero order contact forms $\upi$, but rather the linear combination of these with respect to the $A$ from (\ref{Adx}). Before going ahead with those calculations we establish some basic properties of the space $\G_1$. Notice that for solutions, $\p$, of $\p^*\ebo=g\ebo$, we have
\begin{alignat}{2}
&{} & \p^*\ebo&=g\ebo\\
&\iff & \p^*(A(x)d\xbf)&=g A(x)d\xbf\\
&\iff\quad & A(X)d\mathbf{X}&=g A(x)d\xbf\\
&\iff\quad & A(X)\nabla Xd\xbf&=g A(x)d\xbf,
\end{alignat}
where $\nabla X$ is the Jacobian of $\p$. Solving for $g$, we find that on $\G_1$,
\beq\label{eq:g}
g=A(X)\nabla X A(x)^{-1},
\eeq
and as mentioned previously, the group parameters of $G$ serve to parametrize $\G_1$. Let us write the structure equations of $\ebo$ as
\[
d\eta^i=\sum_{j<k}B^i_{jk}(x)\eta^j\wedge\eta^k,
\]
or, writing $B_{jk}(x)$ for the column vector with $B^i_{jk}(x)$ in the $i\th$ entry,
\[
d\ebo=\sum_{j<k}B_{jk}(x)\eta^j\wedge\eta^k.
\]
On solutions of $\p^*\ebo=g \ebo$, by taking $d$ on both sides and using $\p^*d=d\p^*$, we have
\beq\label{eq:dA}
\begin{aligned}
\p^*d\ebo&=\sum_{j<k}B_{jk}(X)\p^*\eta^j\wedge\p^*\eta^k\\
\iff~~~d(A(X)d\mathbf{X})&=\sum_{j<k}B_{jk}(X)(g\ebo)^j\wedge(g\ebo)^k.
\end{aligned}
\eeq

Now, having established the main properties of $\G_1$, let us compute
\[
d\left(A(X)\pmb{\Upsilon}\right),
\]
where $\pmb{\Upsilon}$ is the column vector of $\Upsilon^1,\ldots, \Upsilon^n$ and $\up^i=X^i_jdx^j$, or, written succinctly, $\pmb{\up}=d\mathbf{X}-\nabla Xd\mathbf{x}$. We find that, on $\G_1$, since $A(X)\nabla X=gA(x)$,
\beq\label{sys:dA}
\begin{aligned}
d\left(A(X)\pmb{\Upsilon}\right)&=d\left(A(X)d\mathbf{X}-A(X)\nabla X d\xbf\right)\\
&=d\left(A(X)d\mathbf{X}\right)-d(gA(x)d\xbf)\\
&=d\left(A(X)d\mathbf{X}\right)-d(g\ebo).
\end{aligned}
\eeq
Now, on solutions, according to (\ref{eq:dA}), this is equal to 
\beq\label{sys:dAup}
\sum_{j<k}B_{jk}(X)(g\ebo)^j\wedge(g\ebo)^k-d(g\ebo).
\eeq

A basis for right invariant forms for the Lie group $G$ can be found in the entries of $dg\cdot g^{-1}$, i.e. we choose $r$ linearly independent entries of $dg\cdot g^{-1}$ as a basis for the right invariant forms on $G$. Let us denote these forms by $\alpha^1,\ldots, \alpha^r$, and so each component of $dg\cdot g^{-1}$, being an invariant 1-form on $G$, is a \emph{constant coefficient} linear combination of the $\alpha^\kappa$, $1\leq\kappa\leq r$. Let us write $dg\cdot g^{-1}=E(\alpha^1,\ldots, \alpha^r)$, where the $(i,j)$ entry in $E$ is $\sum_\kappa F^{ij}_\kappa\ak$. Notice that we can write
\[
d(g\ebo)=dg\wedge\ebo+gd\ebo=dg\cdot g^{-1}\wedge g\ebo + C(x,g,\ebo),
\]
where $C(x,g,\ebo)$ is a vector of two forms with $i\th$ component
\[
\sum_{j<k}C^i_{jk}(x,g)(g\ebo)^j\wedge(g\ebo)^k.
\]
Collecting all of this, (\ref{sys:dAup}) gives that $d(A(X)\mathbf{\Upsilon})$ is equal to
\beq\label{eq:horg}
\sum_{j<k}B_{jk}(X)(g\ebo)^j\wedge(g\ebo)^k- (E\wedge g\ebo)-\sum_{j<k}C_{jk}(x,g)(g\ebo)^j\wedge(g\ebo)^k.
\eeq

The vector space of differential 1-forms on $\Jinf(\R^n\times\R^n)$ is a \emph{direct sum} of horizontal forms, whose space is generated by $dx^1,\ldots, dx^n$, and the contact forms. Since the $a_1,\ldots, a_r$ parametrize the space of 1-jets of $\G_1$, when we write each $\alpha^\kappa$ as a sum of a horizontal and a contact form (where we choose $g\ebo$ as a basis for horizontal forms)
\beq\label{eq:alphac}
\alpha^\kappa=\alpha^\kappa_h+\alpha^\kappa_c=\sum_jz^{\kappa}_{j}(g\ebo)^j+(\alpha^{\kappa}-\sum_jz^{\kappa}_{j}(g\ebo)^j),
\eeq
the contact forms $\alpha^{\kappa}_c:=\alpha^{\kappa}-\sum_jz^{\kappa}_{j}(g\ebo)^j$ are a basis for the contact forms on $\G_1$ and may therefore be used for the test of involution (at the second order) in Algorithm \ref{ckalgo}. Therefore, step (a) of Algorithm \ref{ckalgo} calls for us to replace each $\alpha^k$ in (\ref{eq:horg}) by $\sum_jz^{\kappa}_{j}(g\ebo)^j$, collecting all purely horizontal coefficients, equating the whole thing to zero and solve for as many of the second order parameters, $z^\kappa_j$ as possible. The number of the $z^\kappa_j$ that we do not manage to solve for is, just like before, the fiber-dimension of $\G_{1,1}\to\G_1$, or $r^2$.

Notice that the parameters $z^\kappa_j$ are, just like the $a^\kappa$'s were, very complicated expressions in the 2-jets of $\p\in\G$. But just as before, we do not mind this, as we are sure that they, along with $(x,X,g)$, parametrize the prolonged space $\G_{1,1}$. 

When (\ref{eq:horg}) is written out, in the coordinates $(x,X,g,z)$, and equated to zero we obtain equations of the form
\beq\label{eq:secg}
B^i_{jk}(X)=\sum_\kappa\left(F^{ik}_\kappa z^\kappa_j-F^{ij}_\kappa z^\kappa_k\right)+C^i_{jk}(x,g),
\eeq
where $\sum_\kappa F^{ij}_\kappa\alpha_\kappa$ is the $(i,j)$ entry in $E$. The equations (\ref{eq:secg}) are the relevant second order equations in the first order prolongation of $\G_1$, $\G_{1,1}$. The reader familiar with Cartan's equivalence method will have noticed that the above procedure is \emph{not} exactly that of Cartan's. Namely, in accordance with Algorithm \ref{ckalgo} we obtained the prolonged equations (\ref{eq:secg}) but Cartan's equivalence method sets the left hand sides of (\ref{eq:secg}) to any convenient constant(!), usually zero, and solves the system
\beq\label{eq:carg}
\text{constant}=\sum_\kappa\left(F^{ik}_\kappa z^\kappa_j-F^{ij}_\kappa z^\kappa_k\right)+C^i_{jk}(x,g).
\eeq
Let us denote the corresponding space obtained in this ``normalized'' way $\Gt_{1,1}$. So what is the difference between the two methods, and why should Cartan's route make sense, i.e. why should we be allowed to change parts of the prolonged equations for $\G$ without sabotaging our quest for an involutive form of these equations? This is a historically tricky issue (alluded to in the quotation at the beginning of this section), but the groupoid language of Lie pseudo-groups presented in the last subsection will help explain away this troublesome issue.

\end{subsection}

\begin{subsection}{Invariants}

Recall that an element, $\p$, of a Lie pseudo-group $\G$ of local diffeomorphisms of $\R^n$ can act on the groupoids $\G_p$, $0\leq p\leq\infty$, via the left and right groupoid multiplications (whenever the following groupoid products are defined),
\begin{align*}
L_\p\cdot j^p\psi\resx &=j^p\p\restrict{X}\cdot j^p\psi\resx\\
R_\p\cdot j^p\psi\resx &= j^p\psi\resx\cdot j^p\p^{-1}\restrict{\p(x)}.
\end{align*}
In the coordinates $(x,X,g)$ on $\G_1$, these operations, for $\p^*\ebo=h\ebo$, become
\beq\label{eq:LR}
\begin{aligned}
L_\p\cdot(x,X,g)&=(x,\p(X),hg)\\
R_\p\cdot(x,X,g)&=(\p(x), X, gh^{-1}).
\end{aligned}
\eeq
The defining equations for a Lie pseudo-group are obviously invariant under both actions. 

\begin{center}{\textbf{Reduction of structure group}}\end{center}

Consider the two equations (\ref{eq:secg}) and (\ref{eq:carg}). Since they are affine in the second order coordinates $z^\kappa_i$ we will be able to solve for equally many $z$'s in both equations and therefore we find the same number of second order parametric derivatives, $r^2$. Say, after solving the equations (\ref{eq:secg}) for \emph{all possible} $z$'s, some of the remaining equations are non-trivial (i.e. 0=0) and first order. By the structure of (\ref{eq:secg}) we can see that those equations will have the form
\beq\label{eq:HJ}
H(x,g)=J(X),
\eeq
where $H$ is a vector of constant coefficient linear combinations of the $C^i_{jk}$ and $J$ is a vector of constant coefficient linear combinations of the $B^i_{jk}$. Notice that since the identity map solves this equation, we have
\beq\label{eq:HI}
H(x,I)=J(x).
\eeq
If $H$ explicitly depends on group parameters $a_\kappa$, and is full rank in these parameters, we must solve for as many as we can and begin step (a) of Algorithm \ref{ckalgo} afresh. On the other hand, if we solve the corresponding ``normalized'' equations in (\ref{eq:carg}) we wind up with the expressions
\beq\label{eq:H}
H(x,g),
\eeq
for the \emph{same} functions $H$. Note that, from (\ref{eq:H}), we can recover (\ref{eq:HJ}) via (\ref{eq:HI}). The actual defining eqauations for the pseudo-group (\ref{eq:HJ}) are satisfied by all $\p$ such that $\p^*\ebo=g\ebo$, $g\in G$ and since the defining equations of any Lie pseudo-group are invariant under the operations (\ref{eq:LR}) we have, by acting on (\ref{eq:HJ}) by $j^1\p\restrict{X}=(X,\p(X), h)$ via the left groupoid operation, that
\[
H(x,hg)=J(\p(X)).
\]
In particular, if $H(x,g)=H(\xb, \bar{g})=J(X)$, we have
\beq\label{eq:welldef}
H(x,hg)=H(\xb, h\bar{g})=J(\p(X)).
\eeq
This means that, due to the invariant nature of the expressions $H(x,g)$, there is a well defined action by the \emph{Lie group} $G$ on the range of $H$, $\N$, given by
\[
h\cdot H(x,g)=H(x,hg).
\] 
This observation (due to \'Elie Cartan) opens the door to a procedure called \emph{reduction of the structure group}, that we may apply instead of solving (\ref{eq:HJ}). However, and unfortunately, this will only work if the action of $G$ on $\N$ is \emph{transitive}, i.e. for any two points $b,c\in\N$, there is an $h\in G$ such that $h\cdot b=c$. To describe this procedure, choose an arbitary point $b\in\N$, and a smooth map $g:\R^n\to G$ such that 
\beq\label{eq:Ha}
H(x,g(x))=b,\quad\text{for all}~x\in\R^n.
\eeq
Let $\p\in \G$ so $\p^*\ebo=h\ebo$, for some $h\in G$, and notice that applying $R_\p$ to both sides in (\ref{eq:HJ}) results in (the target coordinate $X$ is $R$-invariant) 
\[
H(\p(x), gh^{-1})=J(X).
\]
This means that $H(x,g)$ is \emph{$R$-invariant} and applying $R_\p$ to (\ref{eq:Ha}) gives
\[
H(\p(x), g(x)h^{-1})=b.
\] 
But $g(x)$ is defined for all $x$ and we have
\[
H(\p(x), g(x)h^{-1})=b=H(\p(x), g(\p(x))).
\]
In terms of the action of $G$ on $\N$ this means that 
\beq\label{eq:gph}
g(\p(x))hg(x)^{-1}
\eeq
is in the \emph{isotropy} group of $b\in\N$ under the action of $G$. Denote this group by $G_b$ and notice that 
\beq\label{eq:Ga}
\p^*(g(x)\ebo)=g(\p(x))hg(x)^{-1}(g(x)\ebo)=\bar{h}(g(x)\ebo),\quad \bar{h}\in G_a.
\eeq
Therefore, by changing base coframes from $\ebo$ to $g(x)\ebo$ we can incorporate the first order equations (\ref{eq:HJ}) as a new equivalence problem with a reduced structure group $G_a\subset G$. But is this new equivalence problem equivalent to $\G^{(1)}_{1,1}$, i.e. do solutions to (\ref{eq:Ga}) automatically solve (\ref{eq:HJ})? To show that this is indeed the case, let $\p$ be such that $g(\p(x))hg(x)^{-1}\in G_b$, where $\p^*\ebo=h\ebo$ and let $H(x,\gb)=b$. Note that this means that $\p$ is a solution to (\ref{eq:Ga}). We shall show that such a $\p$ preserves all level sets of $H$, under $R_\p$, and then it will easily follow that $\p$ solves (\ref{eq:HJ}). So we let $b=H(x,\gb)$ and aim to show that $H(\p(x), \gb h^{-1})=b$. Since $b=H(x,\gb)=H(x,\gb g(x)^{-1}g(x))$, we have that 
\beq\label{eq:gb}
\gb g(x)^{-1}\in G_b.
\eeq
Also, acting on $H(x,\gb)$ by $R_\p$ gives
\begin{align*}
H(\p(x), \gb h^{-1})&=H(\p(x), \gb g(x)^{-1} g(x) h^{-1})\\
&= H(\p(x), \left(\gb g(x)^{-1}\right)\cdot\left( g(x) h^{-1} g(\p(x))^{-1}\right)\cdot g(\p(x))),
\end{align*}
but, both $\gb g(x)^{-1}$ and $g(x) h^{-1} g(\p(x))^{-1}$ are elements of $G_b$ and, by definition, $H(\p(x), g(\p(x)))=b$, and so
\[
H(\p(x), \gb h^{-1})=H(\p(x), g(\p(x)))=b.
\]
So a solution to (\ref{eq:Ga}) leaves the level set $H=b$ invariant. If $H(x,\gb)=c\in\N$, we first choose an element $u\in G$ such that $u\cdot c=b$ (this is where we need transitivity of $G$ on $\N$), so that $H(x,u\gb)=b$. Invariance of the level set $H=b$ implies
\[
b=H(\p(x), u\gb h^{-1})~~\Rightarrow~~c=u^{-1}b=H(\p(x), \gb h^{-1}).
\]
This proves that any solution to (\ref{eq:Ga}), preserves the level sets of $H$ under the right groupoid multiplication (\ref{eq:LR}) and hence that $H$ is invariant under this action. This gives, for solutions to (\ref{eq:Ga}), that
\[
H(x,I)=H(\p(x), h^{-1}),
\]
where $\p^*\ebo=h\ebo$. For $\p^{-1}$ this becomes
\beq\label{eq:HpI}
H(\p(x), I)=H(x,h),
\eeq
which, in view of (\ref{eq:HI}) is the same as (\ref{eq:HJ}).

Ok, this was quite a bit of work, but we have established that adjoining the integrability conditions (\ref{eq:HJ}) to our original system $\p^*\ebo=g\ebo$ gives a system equivalent to the equivalence problem $\p^*\tilde{\ebo}=g\tilde{\ebo}$, where $\tilde{\ebo}$ is a new coframe of the base $\R^n$ and $g$ is an element of the reduced structure group $G_b$. Very well, we now repeat this process, for the new system until no new first order equations, or, equivalently, no invariants $H(x,g)$ are found. But recall that all this is contingent upon the invariants $H(x,g)$ being full rank in the group parameters \emph{and} the action of the structure group $G$ on the range of $H$ being transitive. If the invariants $H(x,g)$ are not full rank then normalizing a sub-collection of these will provide an invariant \emph{not} dependent on a group parameter, $F(x)$. Such a \emph{genuine invariant} of the problem prevents the above procedure of reduction of structure group as we are unable to deduce (\ref{eq:welldef}), since, in this case, this only holds for $x$ and $\xb$ such that $F(x)=F(\xb)$ and the action of the structure group on the range of $H$ is no longer well defined. Systems $\p^*\ebo=g\ebo$ that satisfy these regularity conditions at each step of Cartan's equivalence method are called \emph{constant type} problems, \cite{Gardner89}. 

\bex\label{ex:inv2}
Continuing Example \ref{ex:inv}, but skipping the details of the calculations, after computing the right hand sides in (\ref{eq:carg}) and setting some of them to zero, one of the expressions reduces to
\[
H(x,g)=-\frac{a_4^2}{a_1L_{pp}}.
\]
Other non-normalized expressions become trivial. This invariant is obviously regular in that it is full rank in the group parameters and the action of $G$ on its range is transitive. Setting it equal to $-1$ gives
\[
a_1=\frac{a_4^2}{L_{pp}}
\]
and we choose $a_4=1$, $a_1=\frac{1}{L_{pp}}$ and $a_2=a_3=a_5=0$. Now, for two elements $g,h\in G$, parametrized by $a$'s and $b$'s, respectively, we have
\[
H(x,hg)=-\frac{(b_4a_4)^2}{(b_1a_1)L_{pp}}
\]
and if $g$ satisfies $\dis H(x,g)=-\frac{a_4^2}{a_1L_{pp}}=-1$, $H(x,hg)=-1$ is equivalent to
\[
-\frac{b_4^2}{b_1}=-1\quad\iff\quad b_1=b_4^2.
\]
The group elements of $G$ satisfying the above equation form the isotropy subgroup of $-1$. This reduction of structure group has resulted in a new system $\p^*(g(x)\ebo)=h\left(g(x)\ebo\right)$, where the new coframe is
\[
g(x)\ebo=\bbm \frac{1}{L_{pp}} & 0 & 0 \\ 0 & 1 & 0 \\ 0 & 0 & 1\ebm\cdot\bbm 
dx \\  du-pdx \\ -\tilde{E}dx+L_{pp}dp\ebm
=
\bbm \frac{1}{L_{pp}}dx \\  du-pdx \\ -\tilde{E}dx+L_{pp}dp\ebm,
\]
and $h\in G_{-1}$, where
\[
G_{-1}=\left\{\bbm b_4^2 & b_2 & b_3 \\ 0 & b_4 & 0 \\ 0 & b_5 & \dfrac{1}{b_4}\ebm~ \middle|~b_4\neq 0 \right\}.
\]

\eex

\begin{center}{\textbf{Involution}}\end{center}

If no invariants, genuine or first order, are found during the first step of Cartan's method we can test for involution of the system. For this step we use the basis of contact forms $A(X)\mathbf{\Upsilon}$, instead of the standard basis $\mathbf{\Upsilon}$, as before. The reduced Cartan characters are computed by maximizing ranks of collections of one forms as in (\ref{eq:sqk}) but since we maximize these ranks over all possible linear combinations of horizontal forms, we can just as well work with the horizontal basis consisting of 
\[
\frac{\partial}{\partial (g\ebo)^i},
\]
dual to $g\ebo$, instead of the standard frame $\dfrac{\partial}{\partial x^j}$. At each point in $\G_1$ the frame element $\dfrac{\partial}{\partial (g\ebo)^i}$ is an invertible linear combination of the standard frame elements $\dfrac{\partial}{\partial x^j}$ and the corresponding \emph{total derivative operators} are the same linear combinations of the $D_{x^j}$. Denote these total derivative operators by $\dis D_{(g\ebo)^i}$. The first reduced Cartan character is the maximal rank of the system
\beq\label{eq:1st}
\gamma_1\left(\sum_l a_1^l D_{(g\ebo)^l}\interior d(A(X)\mathbf{\Upsilon}) \right),
\eeq
where $\gamma_1$ denotes projection onto the space of first order contact forms. Notice that during the first step of Cartan's algorithm we computed, in (\ref{eq:horg}), the structure equations $d(A(X)\mathbf{\Upsilon})$ so all the hard work has already been done. Consider (\ref{eq:horg}), and denote, as before, the entries of the matrix $E$ by $\sum_\kappa F^{ij}_\kappa\alpha^\kappa$. Since, in (\ref{eq:1st}) we project onto the space generated by $\alpha^\kappa$, we find that
\begin{align*}
&{}~&&\gamma_1\left( D_{(g\ebo)^l}\interior d(A(X)\mathbf{\Upsilon})^i \right)\\
&=&&\gamma_1\left(D_{(g\ebo)^l}\interior \left(\sum_{j<k}B^i_{jk}(X)(g\ebo)^j\wedge(g\ebo)^k- (E\wedge g\ebo)^i-\sum_{j<k}C^i_{jk}(x,g)(g\ebo)^j\wedge(g\ebo)^k \right)\right)\\
&=&&\gamma_1\left(D_{(g\ebo)^l}\interior \left(- (E\wedge g\ebo)^i \right)\right)\\
&=&&\gamma_1\left(D_{(g\ebo)^l}\interior \left(- \sum_j\sum_\kappa F^{ij}_\kappa\alpha^\kappa\wedge(g\ebo)^j \right)\right)\\
&=&&\sum_\kappa F^{il}_\kappa\alpha^\kappa.
\end{align*}
And so the rank of (\ref{eq:1st}) is equal to the rank of the collection
\beq\label{eq:Fcon}
\left\{\sum_l\sum_\kappa a^l_1F^{il}_\kappa\alpha^\kappa\right\}_{1\leq i\leq n}.
\eeq
This is the same system derived by Cartan. We have then shown that Cartan's equivalence method finds the same number of parametric derivatives of second order, $r^{2}$, as step (a) of Algorithm \ref{ckalgo}, and the same reduced Cartan characters as step (b) of Algorithm \ref{ckalgo}, as applied to the differential equation $\p^*\ebo=g\ebo$.  

\begin{Rem}
Notice that since the coefficients $F^{ij}_\kappa$ are all \emph{constant}, the reduced Cartan characters will also be constant and the regularity condition of Definition \ref{regular} will automatically be satisfied.  
\end{Rem} 

\bex
Continuing Example (\ref{ex:inv2}), and, once again, skipping the explicit calculations, during the second loop through the equivalence method, now with the reduced structure group $G_{-1}$, we find no first order invariants. Counting the number of non-normalized second order parameters we find that $r^2=5$. Looking to set up the system (\ref{eq:Fcon}) we, first of all have,
\[
dg\cdot g^{-1}=\bbm 2\alpha^4 & \alpha^2 & \alpha^3 \\ 0 & \alpha^4 & 0 \\ 0 & \alpha^5 & -\alpha^4 \ebm,
\]
where, for example $\alpha^4=\dfrac{da_4}{a_4}$ but the formulas for the other forms are a little more complicated and not important. For a vector $v=(v_1,v_2,v_3)$ the system (\ref{eq:Fcon}) is
\[
2v_1\alpha^4 +v_2\alpha^2+v_3\alpha^3,\quad v_2\alpha^4,\quad v_2\alpha^5-v_3\alpha^4.
\]
Maximizing this rank over all vectors $v$ is an easy task, and we find the maximal rank $s^1_1=3$ for the choice $v=(0,1,0)$. For the second reduced Cartan character, the choice $v=(0,0,1)$ gives $s^1_2=1$ and since there are only $4$ one-forms we have $s^1_3=0$. Cartan's test is satisfied,
\[
5=r^2=1s^1_1+ 2s^1_2+3s^1_3=1\cdot 3+2\cdot 1 +3\cdot 0=5,
\] 
and the system is involutive at order one. This means that the symmetry group of a non-degenerate (and real-analytic) Lagrangian (i.e. $L_{pp}\neq0$) is infinite dimensional and the general pseudo-group element depends on three arbitrary functions of one variable and one arbitrary function of two variables. Moreover, using Cartan's technique of the graph (see \cite[p. 460-471]{olver95}) this implies that \emph{any} two non-degenerate Lagrangians are equivalent and the general equivalence map will depend on three arbitrary functions of one variable and one arbitrary function of two variables.  
\eex

\begin{center}{\textbf{Prolongation}}\end{center}

If the involutivity test fails, Algorithm \ref{ckalgo} repeats step (a) using the prolonged equation $\G_{1,1}$, instead of $\p^*\ebo=g\ebo$, and computes the exterior derivatives of all first order contact forms
\[
d\Upsilon^i_j\restrict{\G_{1,1}}, 
\]
and sets the purely horizontal parts to zero. (Recall that Cartan's method solves the ``normalized'' equations (\ref{eq:carg}) whereas Algorithm \ref{ckalgo} solves (\ref{eq:secg}).) If the involutivity test fails, Cartan's method rewrites the problem as that of equivalence of a \emph{new} $G$-structure. We now describe this process and show that it is equivalent to what Algorithm \ref{ckalgo} does. First we must collect the appropriate properties of the second order jet space parametrized by $(x,X,g,z)$.


Consider the action of $\p\in\G$ on the space $\G_1$, parametrized by $(x,X,g)$ given by the \emph{right} groupoid operation,
\beq\label{eq:gh-1}
R_\p\cdot(x,X,g)=(\p(x), X, gh^{-1}),
\eeq
where $\p^*\ebo=h\ebo$. 

\begin{Rem}\label{rem:noX}
Notice that $R_\p$ does not act on the target coordinate $X$ and so we could view $R_\p$ as acting on the space of $(x,g)$. This is what is happening in Cartan's treatment.
\end{Rem}

The 1-forms $\alpha^\kappa$ are entries in the matrix of forms $dg\cdot g^{-1}$ and notice that, by (\ref{eq:gh-1}), we have (recall (\ref{eq:alphac})),
\beq\label{eq:Rpdg}
R_\p^*\left(dg\cdot g^{-1}\right)=d(gh^{-1})\cdot (hg^{-1})=dg\cdot g^{-1}+\text{horizontal forms},
\eeq
and in particular
\beq\label{eq:Rpa}
R_\p^*\alpha^\kappa=\alpha^\kappa+\text{horizontal forms}.
\eeq
The horizontal component of the above equations stems from the exterior derivative hitting $h^{-1}$ in (\ref{eq:Rpdg}), whose entries are (complicated) expressions depending on the 1-jets, $j^1\p$, of $\p$. Since $R_\p$ is a \emph{contact transformation}, i.e. it preserves the space of contact forms, (\ref{eq:Rpa}) implies
\[
R_\p^*\alpha^\kappa_c=\alpha^\kappa_c.
\]
To better emphasize which space we are acting on, we shall in what follows write $\Rtwo$ for the action of $R_\p$ on $\G_{1,1}$ parametrized by $(x,X,g,z)$, and $\Rone$ for the corresponding action on $\G_1$. The above equation then means
\[
\Rtwo^*\alpha^\kappa_c=\alpha^\kappa_c.
\]
Analyzing this further, we have
\begin{alignat*}{2}
&{} & \Rtwo^*\alpha^\kappa_c&=\alpha^\kappa_c\\
&\iff\quad & \Rtwo^*\alpha^\kappa-Z^\kappa_j(g\ebo)^j&=\alpha^\kappa-z^\kappa_j(g\ebo)^j\\
&\iff\quad & \Rtwo^*\alpha^\kappa&=\alpha^\kappa+(Z^\kappa_j-z^\kappa_j)(g\ebo)^j,
\end{alignat*}
where $\Rtwo\cdot(x,X,g,z)=(\p(x), X, gh(x)^{-1}, Z)$.


The section $$x\mapsto j^2\p\resx=(x,\p(x), h(x), z(x))$$ annihilates the contact forms $\ak_c=\ak-z^\kappa_j(g\ebo)^j$ so we have
\beq\label{eq:j2p}
j^2\p^*\ak=z^\kappa_j(x)(h(x)\ebo)^j.
\eeq
Also, the identity section of $\G_1$, $\one\resx^{(1)}=(x,x,I)$, pulls $\ak$ back by
\[
\left(\one^{(1)}\right)^*\ak=dI\cdot I^{-1}=0,
\]
and so, it prolongs to
\[
x\overset{\one^{(2)}}{\mapsto} (x,x,I,0).
\]
Now note that 
\[
\one^{(2)}\resx=\Rtwo\comp j^2\p\resx,
\]
and so we have, using (\ref{eq:j2p}), 
\begin{alignat*}{2}
&{} & (\Rtwo\comp j^2\p)^*\ak&=\left(\one^{(2)}\right)^*\ak\\
&\iff\quad  & j^2\p^*\Rtwo^*\ak&=0\\
&\iff\quad  & j^2\p^*(\alpha^\kappa+(Z^\kappa_j-z^\kappa_j)(g\ebo)^j)&=0\\
&\iff\quad  & \left(z^\kappa_j(x)+(Z^\kappa_j-z^\kappa_j)\right)(h(x)\ebo)^j&=0\\
&\iff\quad  & \left(z^\kappa_j(x)+(Z^\kappa_j-z^\kappa_j)\right)&=0\\
&\iff\quad  & Z^\kappa_j&=z^\kappa_j-z^\kappa_j(x).
\end{alignat*}

The above shows that if $j^2\p$ is the section $$x\mapsto (x,\p(x), h(x), z(x))$$ then
\[
\Rtwo\cdot(x,X,g,z)=\left(\p(x), X, g\left(h(x)\right)^{-1}, z-z(x)\right),
\] 
which implies, by acting on the identity section, that $j^2\p^{-1}\restrict{\p(x)}=(\p(x), x, h(x)^{-1}, -z(x))$. Importantly, we also obtain 
\beq\label{eq:Ralpha}
\Rone^*\alpha^\kappa=\alpha^\kappa-(Z^\kappa_j-z^\kappa_j)(g\ebo)^j=-z^\kappa_j(x)(g\ebo)^j.
\eeq
On the other hand, left multiplication is given by
\[
L_{j^2\p}\cdot (x,X,g,z)=\left(x,\p(X), h(X)g, z(X)+z\right).
\]

\begin{Rem}
It is rather remarkable that the groupoid actions work by addition/subtraction in the $z$ variables. This is, of course, different from the action in the standard jet coordinates where this action is affine. This abelian (Lie group) action allows us to continue the \emph{geometric} analysis of the system for $\G$. 
\end{Rem}

Recall that Algorithm \ref{ckalgo}, when computing the space $\G_{1,1}$ solved (\ref{eq:secg}), while Cartan's method computes the space $\Gt_{1,1}$ by solving (\ref{eq:carg}). We are assuming that no lower order invariants are found and hence that these systems can be completely solved. We solve for some $z$'s (call these principal) in terms of other $z$'s (call these parametric, the number of which being $r^2$). For the system (\ref{eq:carg}) we arrive at a solution (or an equation equivalent to (\ref{eq:carg})) we choose to write in the slightly unorthodox form
\beq\label{eq:Mv}
z=\pmb{P}z + \pmb{Q}c(x,g),
\eeq
where $z$ is a vector of the $z^\kappa_j$'s, $\pmb{P}$ and $\pmb{Q}$ are constant coefficient matrices whose rows we denote $P^\kappa_j$ and $Q^\kappa_j$, respectively, and $c(x,g)$ is a vector of some of the $C^i_{jk}(x,g)$ in (\ref{eq:carg}). In (\ref{eq:Mv}) the equation for a parametric $z^\kappa_j$ is simply $z^\kappa_j=z^\kappa_j$ while, if $z^\kappa_j$ is principal, $P^\kappa_j\boldsymbol{\cdot} z$ does not involve any other principal $z$'s (where $\boldsymbol{\cdot}$ is the Euclidean inner product). The reason for the specific form of (\ref{eq:Mv}) becomes clear ina  few paragraphs. For (\ref{eq:secg}) we find a solution (or an equation equivalent to (\ref{eq:secg}))
\beq\label{eq:Nb}
z=\pmb{P}z + \pmb{Q}\left(c(x,g)-b(X)\right),
\eeq
where $b$ is a vector of some of the $B^i_{jk}(X)$ that match the $C^i_{jk}$ present in $c(x,g)$. We note that since the identity map solves (\ref{eq:secg}) and since $j^2\one=(x,x,I,0)$, we have
\beq\label{eq:BeqC}
B^i_{jk}(x)=C^i_{jk}(x,I).
\eeq

The next step in Cartan's method is \emph{absorption of torsion}, i.e. we plug in the solution (\ref{eq:Mv}) into all the contact forms
\[
\alpha^\kappa_c=\alpha^\kappa-z^\kappa_j(g\ebo)^j
\]
to obtain one forms
\[
\alpha^\kappa-(Q_j^\kappa\boldsymbol{\cdot} c(x,g))(g\ebo)^j-(P_j^\kappa\boldsymbol{\cdot} v)(g\ebo)^j,
\]
where the $Q$'s and $P$'s are \emph{rows} in the matrices $\pmb{Q}$ and $\pmb{P}$, respectively, and $\boldsymbol{\cdot}$ is the inner product. Cartan then defines the forms
\[
\pi^\kappa:=\alpha^\kappa-(Q_j^\kappa\boldsymbol{\cdot} c(x,g))(g\ebo)^j,
\]
on $\G_1$ and \emph{prolongs} the original equivalence problem to an equivalence problem on $\G_1$ asking for a map $\Phi:\G_1\to\G_1$ (that fixes the target coordinate $X$, recall Remark \ref{rem:noX}) such that
\begin{align*}
\Phi^*(g\ebo)&=g\ebo,\quad\text{and}\\
\Phi^*\pi^\kappa&=\pi^\kappa+(P_j^\kappa\boldsymbol{\cdot} v)(g\ebo)^j.
\end{align*}

We can rewrite this as
\beq\label{eq:rphi}
\Phi^*\bbm g\ebo \\ \pmb{\pi} \ebm = \bbm I_n & 0 \\ \pmb{M}(v) & I_r \ebm\bbm g\ebo \\ \pmb{\pi} \ebm,
\eeq
where $\pmb{M}(v)$ is the $r\times n$ matrix with entries $P_j^\kappa\boldsymbol{\cdot} v$, $v$ is a vector of the same length as the $P$'s and the $I$'s are identity matrices. The matrices 
\[
\bbm I_n & 0 \\ \pmb{M}(v) & I_r \ebm
\]
form an $r^2$ dimensional \emph{abelian} Lie group, $G^{(2)}$, parametrized by the entries in $v$ that correspond to parametric $z$'s, and (\ref{eq:rphi}) describes a $G^{(2)}$-structure on $\G_1$. The group operation is
\[
\bbm I_n & 0 \\ \pmb{M}(v) & I_r \ebm\bbm I_n & 0 \\ \pmb{M}(y) & I_r \ebm=\bbm I_n & 0 \\ \pmb{P}(v+y) & I_r \ebm,
\]
and so $G^{(2)}\cong\R^{r^2}$. We claim that the equivalence problem (\ref{eq:rphi}) is equivalent to the system of differential equations $\G_{1,1}$ and we are therefore starting this process over at step (a) of Algorithm \ref{ckalgo}.  

\begin{theorem}
The differential equation $\G_{1,1}$, defined by $\p^*\ebo=g\ebo$, $g\in G$, and (\ref{eq:secg}), is equivalent to (\ref{eq:rphi}).
\end{theorem} 

\begin{proof}
First let's prove that $\Phi=R_{j^1\p}$ for some local diffeomorphism $\p$. We have
\[
\Phi(x,X,g)=(\p(x,X,g), X, \psi(x,X,g))
\]
so
\begin{alignat*}{2}
&{} & \Phi^*(g\ebo)&=g\ebo\\
&\iff\quad  & \psi(\p^*\ebo)&=g\ebo\\
&\iff\quad  & \psi( h\ebo)&=g\ebo\\
&\iff\quad  & \psi(x,X,g)&=gh^{-1},
\end{alignat*}
for some $h\in\text{GL}(n)$. The second equation above implies that $\p=\p(x)$ is a function of $x$ only, and so
\[
\Phi(x,X,g)=(\p(x), X, gh^{-1}),\quad\text{where}~~\p^*\ebo=h\ebo,
\]
which means that $\Phi=R_{j^1\p}$. Let
\[
j^2\p\restrict{x}=(x, \p(x), h(x), z(x))~~\Rightarrow~~ j^2\p^{-1}\restrict{\p(x)}=(\p(x), x, h(x)^{-1}, -z(x)).
\]

Next, recall (\ref{eq:Ralpha}) where we found that
\[
\Rone^*\alpha^\kappa=\alpha^\kappa-z^\kappa_j(x)(g\ebo)^j,
\]
and
\[
\Rtwo\cdot(x,X,g,z)=\left(\p(x), X, g\left(h(x)\right)^{-1}, z-z(x)\right).
\]
The condition $\Rone^*\pmb{\pi}=\pmb{\pi}+\pmb{M}(v)(g\ebo)$ can be written
\begin{alignat*}{2}
&{} & \Rone^*\pi^\kappa&=\pi^\kappa+(P^\kappa_j\boldsymbol{\cdot}v)(g\ebo)^j\\
&\iff~~  & \Rone^*\left(\alpha^\kappa-(Q^\kappa_j\boldsymbol{\cdot}c)(g\ebo)^j\right)&=\alpha^\kappa-(Q^\kappa_j\boldsymbol{\cdot}c)(g\ebo)^j+(P^\kappa_j\boldsymbol{\cdot}v)(g\ebo)^j\\
&\iff~~  & \alpha^\kappa-z^\kappa_j(x)(g\ebo)^j-(Q^\kappa_j\boldsymbol{\cdot}c_\tau)(g\ebo)^j&=\alpha^\kappa-(Q^\kappa_j\boldsymbol{\cdot}c_\sigma)(g\ebo)^j+(P^\kappa_j\boldsymbol{\cdot}v)(g\ebo)^j,
\end{alignat*}

where $c_\tau=c\comp j^2\p^{-1}\restrict{x}=c(\p(x), gh(x)^{-1})$ and $c_\sigma=c(x, g)$. The above equations imply that for $z^\kappa_j$ parametric we find
\[
-z^\kappa_j(x)=v^\kappa_j
\]
and hence that $\p^{-1}$ satisfies the second order equations
\[
-z(x)=-\pmb{P}z(x)+\pmb{Q}(c(\p(x), gh(x)^{-1})-c(x,g)). 
\]
Setting $g=I$ this becomes
\[
-z(x)=-\pmb{P}z(x)+\pmb{Q}(c(\p(x), h(x)^{-1})-c(x,I)). 
\]
This means that $\p$ satisfies the second order equation
\[
z(x)=\pmb{P}z(x)+\pmb{Q}(c(x, h(x))-c(\p(x),I))
\]
but since $b(\p(x))=c(\p(x),I)$ (cf. (\ref{eq:BeqC})) we have shown that solutions to (\ref{eq:rphi}) are solutions to $\G_{1,1}$. The converse follows immeadiately from  (\ref{eq:Ralpha}) and the second order determining equations for $\G_{1,1}$, (\ref{eq:Nb}).
\end{proof}

We have thus proven that the prolonged equivalence problem is equivalent to the standard prolongation of the determining equations $\G_1$ to $\G_{1,1}$. We have already shown that, at the first step of both processes, Algorithm \ref{ckalgo} and Cartan's equivalence method compute the same integrability conditions, the number of free jet variables of the next order and reduced Cartan characters. Assuming that all the equivalence problems arrived at during Cartan's method are of constant type, since Algorithm \ref{ckalgo} terminates at involution, so must Cartan's equivalence method. 


\begin{Rem}
One possible outcome of Cartan's equivalence method is when we manage to normalize \emph{all} group parameters at some order of prolongation. This means that the corresponding differential equation is \emph{maximally overdetermined} at that order. After reduction of the structure group we are then faced with a bona-fide coframe on some space whose symmetry group is $\G$. Cartan solved completely the general equivalence problem for coframes and so we can also consider the equivalence problem solved in this case.
\end{Rem}

\begin{theorem}[Termination of Cartan's equivalence method]
For equivalence problems of constant type and for which we can never normalize all group parameters (at each order of prolongation), Cartan's equivalence method terminates at involution.
\end{theorem}

\end{subsection}


\end{section}